\documentclass[11pt]{amsart}

\usepackage[margin=1in]{geometry}
\usepackage[T1]{fontenc}
\usepackage{lmodern,microtype}
\usepackage{amsmath,amssymb,amsthm,mathtools}
\usepackage{booktabs,array}
\usepackage{tikz,pgfplots}
\pgfplotsset{compat=1.18}
\usepackage[hidelinks]{hyperref}

\numberwithin{equation}{section}
\newtheorem{theorem}{Theorem}[section]
\newtheorem{proposition}[theorem]{Proposition}
\newtheorem{lemma}[theorem]{Lemma}
\newtheorem{corollary}[theorem]{Corollary}
\theoremstyle{remark}
\newtheorem{remark}[theorem]{Remark}
\newtheorem*{theoremA}{Theorem A}
\newtheorem*{theoremB}{Theorem B}
\newtheorem*{theoremC}{Theorem C}

\newcommand{\C}{\mathbb C}
\newcommand{\Z}{\mathbb Z}
\newcommand{\g}{\mathfrak g}
\newcommand{\q}{\mathfrak q}
\newcommand{\p}{\mathfrak p}
\newcommand{\s}{\mathfrak s}
\newcommand{\slie}{\mathfrak{sl}}

\newcommand{\rv}{\vee}
\newcommand{\one}{\mathbf 1}
\DeclareMathOperator{\ad}{ad}
\DeclareMathOperator{\diag}{diag}
\DeclareMathOperator{\ind}{ind}

\title[Nonunimodal principal spectra]{Nonunimodal principal spectra of Frobenius parabolic Lie algebras}
\author{Vincent E. Coll, Jr.}
\address{Department of Mathematics, Lehigh University, Bethlehem, Pennsylvania 18015, USA}
\email{vec208@lehigh.edu}
\thanks{ORCID: 0000-0002-5775-5522.}
\date{}


\begin{document}

\begin{abstract}
Strict unimodality of Ooms spectra fails for Frobenius parabolic Lie algebras beyond the maximal-parabolic class.  We construct explicit Frobenius parabolics whose principal multiplicity sequences contain arbitrarily long outward rises: for every $L\ge1$ there are $L$ consecutive negative first differences, each of quadratic size in $L$.  In a simple one-parameter subfamily the transition to nonunimodality occurs in $\slie_{136}$, where two adjacent multiplicities are $541<542$.  We also give a representation-theoretic criterion for unimodality.  If $F$ is a Frobenius functional with semisimple integral principal element $H$ and $\g=\bigoplus_j\g_j$, the Kirillov form makes the centered grading symplectic and its character has the unique expansion
\[
 \Theta_\g(x)=\sum_{m\ge1}(\dim\g_m-\dim\g_{m+1})\,\chi_{2m-1}(x)
\]
in irreducible $SL_2$ characters.  Thus Ooms unimodality is equivalent to effectivity of this virtual $SL_2$ character, or equivalently to extension of the centered principal cocharacter to a symplectic $SL_2$ action.  For Frobenius seaweeds the unbroken-spectrum theorem says that, after centering, every odd weight in the support occurs; strict unimodality asks for the stronger positivity of every irreducible coefficient.  A companion paper proves that stronger property for every Frobenius maximal parabolic.  The families constructed here show that beyond that class there can be arbitrarily long blocks of negative irreducible coefficients, with unbounded total negative mass.
\end{abstract}

\maketitle

\section{Introduction}

Let $\g$ be a finite-dimensional complex Frobenius Lie algebra and let $F\in\g^*$ be Frobenius.  The nondegenerate Kirillov form
\[
 B_F(x,y)=F([x,y])
\]
identifies $\g$ with $\g^*$ and determines the principal element $H=\widehat F$ by
\[
 F([H,x])=F(x)\qquad(x\in\g).
\]
When $H$ is semisimple, its eigenspace decomposition
\[
 \g=\bigoplus_j\g_j,\qquad [H,x]=jx\quad(x\in\g_j),
\]
is the principal grading.  For algebraic Frobenius Lie algebras the spectrum is independent of the Frobenius functional; this is the Ooms spectrum \cite{Ooms,GG}.

The semisimple integral hypothesis is a genuine restriction within the class of Frobenius Lie algebras.  In general a principal element need not be semisimple, and even semisimple principal elements can have nonrational eigenvalues; explicit families are given by Diatta--Manga \cite{DiattaManga}.  By contrast, Gerstenhaber--Giaquinto proved that the principal element of a saturated Frobenius subalgebra of a simple Lie algebra is semisimple and that, for saturated Frobenius subalgebras of $\slie_n$, its eigenvalues are integral and independent of the Frobenius functional \cite{GG}.  Joseph later established integrality for Frobenius biparabolic subalgebras of arbitrary semisimple Lie algebras \cite{Joseph}.  Seaweeds are biparabolic subalgebras containing a Cartan subalgebra, so these integrality results apply to them.

Building on earlier type-$A$ work \cite{CMW,CHM}, Cameron--Coll--Hyatt--Magnant proved that the principal spectrum of every Frobenius seaweed is an unbroken set of integers and that its multiplicities are symmetric about $1/2$ \cite{CCHM}.  The symmetry is also a general consequence of Frobenius duality, already present in Ooms' work \cite{Ooms}; the unbroken-support assertion is the additional seaweed phenomenon that will be essential below.

The remaining shape question was unimodality.  Coll--Magnant--Wang initiated it by conjecturing strict unimodality of the multiplicities \cite[Conjecture~21]{CMW}.  Mayers--Russoniello proved substantial positive families \cite{MR}.  For Frobenius maximal parabolics, Giaquinto--Irving--Lauve--Mastnak proved unimodality by a good-grading argument \cite{GILM}.  In essentially contemporaneous work, by different methods, Coll proved the stronger statement that every Frobenius maximal parabolic of type $A$ has a \emph{strictly} unimodal Ooms multiplicity sequence, together with further log-concavity results \cite{CollMax}.  Thus the full maximal-parabolic class satisfies the strict positivity predicted by the Coll--Magnant--Wang conjecture.

We show that this behavior does not persist for general Frobenius parabolics.  There are explicit infinite families of nonunimodal principal spectra in which the outward rise can be made arbitrarily long and arbitrarily large.

\begin{theoremA}
For every integer $L\ge1$, there is an explicitly constructed Frobenius parabolic with principal multiplicities $h_j$ such that, for every $j=3,\ldots,L+2$,
\begin{equation}\label{eq:strongdefectintro}
 h_j-h_{j+1}\le-\bigl(2(L+5)^2-1\bigr).
\end{equation}
In particular,
\[
 h_3<h_4<\cdots<h_{L+3}.
\]
The total rise over this block is
\begin{equation}\label{eq:massintro}
 h_{L+3}-h_3=\frac{L(5L^2+43L+106)}2.
\end{equation}
For every $e\in\g_1$, each of the $L$ maps
\[
 \ad e:\g_j\longrightarrow\g_{j+1},\qquad j=3,\ldots,L+2,
\]
has cokernel dimension at least $2(L+5)^2-1$.
\end{theoremA}

A particularly simple one-parameter family already gives infinitely many counterexamples.

\begin{theoremB}
For every $r\ge0$, the parabolic
\begin{equation}\label{eq:simplefamily}
 \p_r=\q\bigl((60+4r)\mid(24,5,2^{[r]},1,2^{[r]},30)\bigr)
\end{equation}
is Frobenius, has dimension $8r^2+244r+2550$, and satisfies
\[
 h_3(r)=237+16r,\qquad h_4(r)=181+19r,\qquad h_3(r)-h_4(r)=56-3r.
\]
Its Ooms spectrum is strictly unimodal for $0\le r\le18$ and is not unimodal for every $r\ge19$.  At $r=19$ the algebra lies in $\slie_{136}$, has dimension $10{,}074$, and its positive-degree multiplicities are
\[
 (1659,891,541,542,485,379,268,170,82,20).
\]
\end{theoremB}

The first displayed counterexample contains the local reversal $541<542$, while Theorem~A gives arbitrarily long blocks of consecutive outward increases.  The constructions are uniform.  For integers $a,b\ge2$ and $r\ge0$, write $2^{[r]}$ for a list of $r$ parts equal to $2$ and set
\begin{equation}\label{eq:introfamily}
 \p_{a,b,r}
 =\q\bigl((2ab+4r)\mid(a(b-1),a-1,2^{[r]},1,2^{[r]},ab)\bigr).
\end{equation}
Every algebra in \eqref{eq:introfamily} is Frobenius.  For square parameters $a=b=t$, the adjacent differences satisfy
\begin{equation}\label{eq:introgap}
 h_j(t,t,r)-h_{j+1}(t,t,r)
 =6jt-\frac{9j^2+j}{2}+1-4r,
 \qquad3\le j\le t-3.
\end{equation}
Taking $r=t^2$ and $t=L+5$ gives Theorem~A; the specialization $(a,b)=(6,5)$ gives Theorem~B.

There is also a representation-theoretic interpretation of unimodality.  Put $h_j=\dim\g_j$ and define the \emph{centered principal character}
\begin{equation}\label{eq:centeredcharintro}
 \Theta_\g(x)=\sum_j h_jx^{2j-1}.
\end{equation}
If $\chi_n(x)=x^n+x^{n-2}+\cdots+x^{-n}$ denotes the character of $\operatorname{Sym}^n(\C^2)$, the following theorem identifies unimodality with an $SL_2$-extension condition.

\begin{theoremC}
Let $F$ be a Frobenius functional whose principal element $H$ is semisimple with integral eigenvalues.  Then the centered cocharacter
\[
 \lambda_H(x)|_{\g_j}=x^{2j-1}\operatorname{id}
\]
preserves the Kirillov form, and
\begin{equation}\label{eq:weylintro}
 \Theta_\g(x)=\sum_{m\ge1}(h_m-h_{m+1})\chi_{2m-1}(x).
\end{equation}
The following are equivalent:
\begin{enumerate}
\item the principal multiplicities are unimodal;
\item the virtual $SL_2$ character $\Theta_\g$ is effective;
\item $\lambda_H$ extends to a homomorphism $SL_2\to Sp(\g,B_F)$.
\end{enumerate}
Strict unimodality is equivalent to strict positivity of every coefficient in \eqref{eq:weylintro} through the occupied positive support.
\end{theoremC}

The unbroken-spectrum theorem can be restated in these terms.  After centering, its support is the complete odd string
\[
 -(2N-1),-(2N-3),\ldots,-1,1,\ldots,2N-3,2N-1.
\]
If
\[
 p_m=h_m-h_{m+1},\qquad h_{N+1}=0,
\]
then \eqref{eq:weylintro} gives
\begin{equation}\label{eq:tailsumsintro}
 h_j=\sum_{m=j}^{N}p_m.
\end{equation}
Unbrokenness says that every tail sum is strictly positive.  The Coll--Magnant--Wang strict-unimodality conjecture asks for the stronger coefficientwise inequalities $p_m>0$.  Thus earlier work established full weight support and positivity in the weight basis, whereas unimodality asks for positivity in the irreducible $SL_2$ basis.  The examples above separate these two notions.

For the $10{,}074$-dimensional example in Theorem~B, the irreducible expansion is
\begin{equation}\label{eq:136virtualintro}
 \Theta_{\p_{19}}
 =768\chi_1+350\chi_3-\chi_5+57\chi_7+106\chi_9+111\chi_{11}
 +98\chi_{13}+88\chi_{15}+62\chi_{17}+20\chi_{19}.
\end{equation}
Thus the first displayed counterexample leaves the positive $SL_2$ cone by exactly one copy of $\operatorname{Sym}^5(\C^2)$, while Theorem~A produces arbitrarily long blocks of negative irreducible coefficients with unbounded total negative mass.

The grading also interacts with the standard Yang--Baxter structure of a Frobenius algebra.  The inverse Kirillov form $r_F=B_F^{-1}\in\wedge^2\g$ is a nondegenerate triangular solution of the classical Yang--Baxter equation \cite{GGCYBE}.  The principal grading satisfies
\[
 [H,r_F]=r_F,
\]
while the centered cocharacter in Theorem~C fixes $r_F$.  Moreover $\operatorname{ad}_H^*F=-F$, so $H$ generates the radial dilation of the Frobenius functional inside its open coadjoint orbit.  This is the same deformation-theoretic setting in which triangular Lie structures enter twists and universal deformation formulas; see, for example, \cite{GZ}.  Thus the Ooms multiplicities give the weight distribution of the associated conformally symplectic dilation.

The construction uses two operations on meanders.  We first consider a two-parameter family of Frobenius seaweeds whose potential histograms are products of interval polynomials.  We then attach an alternating chain at one endpoint of the meander and fold the resulting path into a parabolic.  Both operations preserve the single-path property.  Endpoint extension changes the potential histogram linearly while affecting only the two central Ooms multiplicities; the fold inserts the autocorrelation of that histogram into a new Ooms spectrum.  Away from the center the first differences therefore vary linearly with the extension parameter.

The paper is organized as follows.  Section~2 develops the centered principal character and proves Theorem~C, then fixes the seaweed and meander conventions.  Section~3 isolates the endpoint-extension and full-fold operations.  Section~4 constructs the rectangular seed family.  Section~5 proves the quantitative statement in Theorem~A.  Section~6 proves Theorem~B and displays the full spectra on both sides of the transition.  Section~7 records boundary questions suggested by the construction.  Appendix~A gives an independent exact Kirillov certificate for the $\slie_{136}$ example; it is not used in the infinite-family proof.

\section{Seaweeds, meanders, and principal spectra}\label{sec:prelim}

All Lie algebras are over $\C$. Let $F\in\g^*$. The Kirillov form is
\[
 B_F(x,y)=F([x,y]),
 \qquad
 \ind\g=\min_{F\in\g^*}\dim\ker B_F.
\]
The algebra is Frobenius when its index is zero. If $F$ is Frobenius, the principal element $\widehat F$ is characterized by
\[
 F([\widehat F,x])=F(x)\qquad(x\in\g).
\]
For the seaweeds considered below, $\ad\widehat F$ is semisimple with integral eigenvalues. Write
\[
 \g_k=\{x\in\g:[\widehat F,x]=kx\},
 \qquad
 H_\g(z)=\sum_{k\in\Z}(\dim\g_k)z^k.
\]
For convenience, we encode these multiplicities in the Laurent polynomial $H_\g$, which we call the \emph{Ooms multiplicity polynomial}.

The defining equation implies $F(\g_k)=0$ unless $k=1$. Since the Kirillov form is nondegenerate, it pairs $\g_k$ perfectly with $\g_{1-k}$. Hence
\begin{equation}\label{eq:symmetry}
 H_\g(z)=zH_\g(z^{-1}).
\end{equation}
Thus unimodality is equivalent to
\[
 [z^j]H_\g\ge[z^{j+1}]H_\g\qquad(j\ge1),
\]
with coefficients understood to be zero outside the support. We use ``strictly unimodal'' to allow the central equality forced by \eqref{eq:symmetry} at eigenvalues $0$ and $1$.

\subsection{The centered principal character and the \texorpdfstring{$SL_2$}{SL2} cone}

The perfect pairing behind \eqref{eq:symmetry} contains more information than numerical symmetry.  The principal identity and Jacobi give, for all $x,y\in\g$,
\begin{equation}\label{eq:conformal}
 B_F([H,x],y)+B_F(x,[H,y])=B_F(x,y).
\end{equation}
Thus $\ad H-\tfrac12 I$ is infinitesimally symplectic.  Since the principal eigenvalues are integral in the situations considered here, define
\begin{equation}\label{eq:centeredcochar}
 \lambda_H(x)|_{\g_j}=x^{2j-1}\operatorname{id}_{\g_j},
 \qquad x\in\C^\times.
\end{equation}
Equation~\eqref{eq:conformal} says precisely that $\lambda_H$ takes values in $Sp(\g,B_F)$.

We encode its character by
\begin{equation}\label{eq:centeredcharacter}
 \Theta_\g(x)=\operatorname{Tr}(\lambda_H(x)|\g)
 =\sum_jh_jx^{2j-1}=x^{-1}H_\g(x^2).
\end{equation}
For $m\ge1$, let
\[
 \chi_{2m-1}(x)=x^{2m-1}+x^{2m-3}+\cdots+x^{-(2m-1)}
\]
be the character of $\operatorname{Sym}^{2m-1}(\C^2)$.

\begin{theorem}[$SL_2$ criterion for unimodality]\label{thm:sl2dictionary}
Let $F$ be Frobenius and suppose its principal element is semisimple with integral eigenvalues.  Put $h_m=0$ beyond the support.  Then
\begin{equation}\label{eq:weylexpansion}
 \boxed{\Theta_\g(x)=\sum_{m\ge1}(h_m-h_{m+1})\chi_{2m-1}(x).}
\end{equation}
The following are equivalent:
\begin{enumerate}
\item $h_1\ge h_2\ge h_3\ge\cdots$;
\item every coefficient in \eqref{eq:weylexpansion} is nonnegative;
\item the symplectic cocharacter $\lambda_H$ extends to a representation $SL_2\to Sp(\g,B_F)$ whose diagonal torus is \eqref{eq:centeredcochar}.
\end{enumerate}
If the positive support is $1,\ldots,N$, strict unimodality is equivalent to $h_m-h_{m+1}>0$ for $1\le m\le N$, where $h_{N+1}=0$.
\end{theorem}

\begin{proof}
Set $d_m=h_m-h_{m+1}$.  The coefficient of $x^{2k-1}$ on the right side of \eqref{eq:weylexpansion} is
\[
 \sum_{m\ge k}d_m=h_k,
\]
and symmetry $h_j=h_{1-j}$ supplies the negative weights.  This proves the expansion and the equivalence of (1) and (2).

Assume (2).  Form the $SL_2$-module
\[
 W=\bigoplus_{m\ge1}d_m\operatorname{Sym}^{2m-1}(\C^2).
\]
Every odd highest-weight irreducible carries an invariant nondegenerate alternating form, so $W$ is symplectic.  Its diagonal torus has the same weight multiplicities as \eqref{eq:centeredcochar}.  The Kirillov form pairs the weight-$2m-1$ space of $\g$ perfectly with the weight-$-(2m-1)$ space and no other weight space.  Choosing isomorphisms on the positive weight spaces and the dual isomorphisms on the negative ones gives a weight-preserving symplectic isomorphism $W\simeq(\g,B_F)$.  Transporting the $SL_2$ action proves (3).

Conversely, if (3) holds, the $SL_2$ representation decomposes into irreducibles.  Since the torus weights of \eqref{eq:centeredcochar} are all odd, only odd highest-weight irreducibles can occur.  Their multiplicities are uniquely the coefficients $d_m$ in \eqref{eq:weylexpansion}, and hence are nonnegative.  The strict statement is immediate.
\end{proof}

\begin{corollary}[Unbrokenness versus Lefschetz positivity]\label{cor:unbrokenlefschetz}
Let $\g$ be a type-$A$ Frobenius seaweed whose positive Ooms support is $1,\ldots,N$, and put $p_m=h_m-h_{m+1}$ with $h_{N+1}=0$.  Then
\[
 \Theta_\g=\sum_{m=1}^{N}p_m\chi_{2m-1},
 \qquad
 h_j=\sum_{m=j}^{N}p_m>0
 \quad(1\le j\le N).
\]
Equivalently, the centered principal character has every odd weight from $-(2N-1)$ to $2N-1$.  Thus the general unbroken-spectrum theorem of Cameron--Coll--Hyatt--Magnant \cite{CCHM}, building on the earlier type-$A$ results \cite{CMW,CHM}, amounts to the strict positivity of all tail sums of the virtual primitive multiplicities, whereas strict unimodality is the stronger condition $p_m>0$ for every $m$.
\end{corollary}

\begin{remark}[Yang--Baxter and coadjoint interpretations]\label{rem:yb}
The uncentered one-parameter action $x|_{\g_j}\mapsto t^jx$ is a Lie-algebra automorphism and scales the Kirillov form by $t$.  Its centered double cover \eqref{eq:centeredcochar} is symplectic.  If $r_F=B_F^{-1}\in\wedge^2\g$ is the nondegenerate triangular $r$-matrix associated to $F$, then
\[
 [H,r_F]=r_F,
\]
so $\lambda_H$ fixes $r_F$; see \cite{GGCYBE} for the Frobenius--CYBE correspondence.  Also $\ad_H^*F=-F$, so the principal element infinitesimally generates radial scaling of $F$ along the open coadjoint orbit.
\end{remark}

Theorem~\ref{thm:sl2dictionary} identifies the Lefschetz condition relevant to unimodality.  Frobenius duality always supplies the centered symplectic torus and the virtual $SL_2$ character.  Unimodality asks whether that virtual character is effective.  In an actual $SL_2$ decomposition, the integers
\[
 p_m=h_m-h_{m+1}
\]
are the multiplicities of highest-weight, or primitive, summands; we therefore regard them as the \emph{virtual primitive multiplicities} of the principal grading.  Negative $p_m$ are exactly the obstruction to $SL_2$ effectivity.  This is the $SL_2$ mechanism used in many classical unimodality theorems \cite{Almkvist,Stanley}.  The $SL_2$ action in Theorem~\ref{thm:sl2dictionary} is an action on the underlying symplectic vector space; it is not asserted to act by Lie-algebra automorphisms.  A good or inner Lefschetz element, when it exists, is therefore additional structure.

Let $A=(a_1,\ldots,a_u)$ and $B=(b_1,\ldots,b_v)$ be compositions of $n$, and let $\alpha(i)$ and $\beta(i)$ denote the $A$-block and $B$-block containing $i$, respectively.  Our convention is
\begin{equation}\label{eq:seaweedconvention}
 \q(A\mid B)
 =\mathfrak h
 \oplus\!\!\bigoplus_{\substack{i>j\\ \alpha(i)=\alpha(j)}}\C E_{ij}
 \oplus\!\!\bigoplus_{\substack{i<j\\ \beta(i)=\beta(j)}}\C E_{ij},
\end{equation}
where $\mathfrak h\subset\slie_n$ is the trace-zero diagonal subalgebra.  Thus the $A$-blocks contribute lower-triangular matrix units and the $B$-blocks contribute upper-triangular matrix units.  This convention is the transpose of that used in \cite{GILM}; transposition gives an isomorphic Lie algebra and preserves the principal multiplicities.

The meander of $\q(A\mid B)$ has vertices $1,\ldots,n$ in order. Within each top block, join the first vertex to the last, the second to the penultimate, and so on by nested arcs above the line. Do the same below the line for the bottom blocks. Dergachev and Kirillov proved that the seaweed is Frobenius precisely when the meander is a single path \cite{DK}.

Orient each top arc from right to left and each bottom arc from left to right. On a single-path meander choose an integer potential $\psi=(\psi_1,\ldots,\psi_n)$ satisfying
\begin{equation}\label{eq:potential}
 \psi_j-\psi_i=1\quad\text{on a top pair }i<j,
 \qquad
 \psi_i-\psi_j=1\quad\text{on a bottom pair }i<j.
\end{equation}
It is unique up to a common translation. The directed-arc functional is Frobenius \cite{CD,MR}, and
\begin{equation}\label{eq:principal}
 \widehat F
 =\diag(\psi_1,\ldots,\psi_n)
 -\frac{\sum_i\psi_i}{n}I_n
\end{equation}
is a principal element. In particular $E_{ij}$ has degree $\psi_i-\psi_j$.

For a set of vertices $S$, let
\[
 U_S(z)=\sum_{i\in S}z^{\psi_i}.
\]
The full vertex histogram $U=U_{\{1,\ldots,n\}}$ is defined only up to a monomial factor, but its autocorrelation $UU^\rv$, where $P^\rv(z)=P(z^{-1})$, is well defined.  This autocorrelation is the term that enters the full-fold construction below.

We shall also use directional block contributions. If $Q$ is a top block of size $m$, put
\[
 H_Q^{\rm top}
 =\left\lfloor\frac m2\right\rfloor
 +\sum_{\substack{i,j\in Q\\i>j}}z^{\psi_i-\psi_j};
\]
for a bottom block use the analogous sum over $i<j$. Summing over all top and bottom blocks gives $H_\g$; the constant terms allocate exactly the $n-1$ trace-zero diagonal directions.

\begin{lemma}\label{lem:evenblock}
If a top block has size $2m$ and its left-half potential histogram is $T$, then its directional contribution is
\[
 (1+z)TT^\rv.
\]
\end{lemma}

\begin{proof}
If the left-half potentials are $t_1,\ldots,t_m$, then the right-half potentials, in vertex order, are $t_m+1,\ldots,t_1+1$. Pairs within the two halves, together with the $m$ allocated zero degrees, contribute $TT^\rv$. Cross pairs contribute $zTT^\rv$.
\end{proof}

\section{Endpoint extension and full fold}\label{sec:amplifier}

The counterexamples are produced by two operations on meanders.  We state the operations first and specialize the seed in Section~\ref{sec:rectangular}.

Suppose
\[
 \s=\q((1,a_2,\ldots,a_u)\mid B)\subseteq\slie_N
\]
is Frobenius. The first vertex is unmatched on top and hence is an endpoint of the path. Let $\theta$ be its potential, let $U$ be the full potential histogram, and let $H$ be the Ooms multiplicity polynomial. Put
\[
 V_\theta=z^{\theta-1}+z^\theta.
\]

Figure~\ref{fig:pathoperations} shows the two path operations schematically.  The endpoint-extension panel displays the case $r=2$; the same alternating two-edge pattern is simply repeated for general $r$.  The full-fold panel isolates one old top edge.  It is this local replacement, repeated simultaneously over all old top edges, that makes the path preservation in Proposition~\ref{prop:fold} visible.

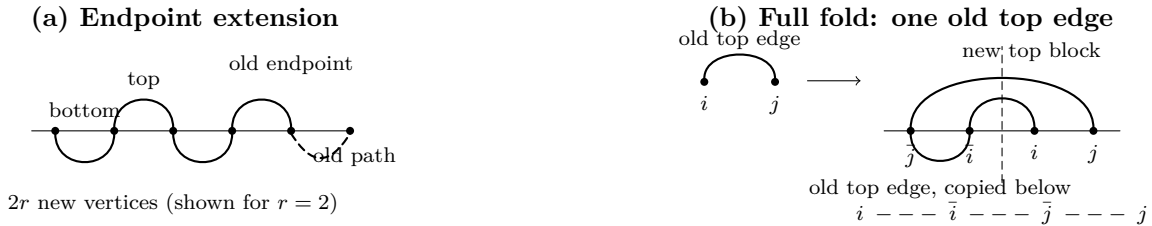
\begin{figure}[htbp]
\centering
\begin{tikzpicture}[x=0.78cm,y=0.72cm,line cap=round,line join=round]
\begin{scope}[xshift=-5.2cm]
  \node[font=\small\bfseries] at (0,2.05) {(a) Endpoint extension};
  \draw[thin] (-2.6,0)--(2.8,0);
  \foreach \x in {-2.2,-1.2,-0.2,0.8,1.8,2.8}{\fill (\x,0) circle (1.5pt);}
  \draw[thick] (-2.2,-0.05).. controls (-2.2,-0.75) and (-1.2,-0.75)..(-1.2,-0.05);
  \draw[thick] (-0.2,-0.05).. controls (-0.2,-0.75) and (0.8,-0.75)..(0.8,-0.05);
  \draw[thick] (-1.2,0.05).. controls (-1.2,0.75) and (-0.2,0.75)..(-0.2,0.05);
  \draw[thick] (0.8,0.05).. controls (0.8,0.75) and (1.8,0.75)..(1.8,0.05);
  \draw[thick,densely dashed] (1.8,-0.05).. controls (2.15,-0.65) and (2.45,-0.65)..(2.8,-0.05);
  \node[font=\scriptsize,anchor=north] at (-0.2,-0.95) {$2r$ new vertices (shown for $r=2$)};
  \node[font=\scriptsize,anchor=south] at (1.8,0.86) {old endpoint};
  \node[font=\scriptsize,anchor=west] at (2.0,-0.48) {old path};
  \node[font=\scriptsize] at (-1.7,0.38) {bottom};
  \node[font=\scriptsize] at (-0.7,0.92) {top};
\end{scope}

\begin{scope}[xshift=4.4cm]
  \node[font=\small\bfseries] at (0,2.05) {(b) Full fold: one old top edge};
  \fill (-3.5,0.9) circle (1.5pt);
  \fill (-2.3,0.9) circle (1.5pt);
  \node[font=\scriptsize,anchor=north] at (-3.5,0.82) {$i$};
  \node[font=\scriptsize,anchor=north] at (-2.3,0.82) {$j$};
  \draw[thick] (-3.5,0.96).. controls (-3.5,1.55) and (-2.3,1.55)..(-2.3,0.96);
  \node[font=\scriptsize] at (-2.9,1.67) {old top edge};
  \draw[->,thin] (-1.75,0.92)--(-0.85,0.92);

  \draw[thin] (-0.45,0)--(3.55,0);
  \draw[densely dashed] (1.55,-0.95)--(1.55,1.55);
  \foreach \x in {0,1,2.1,3.1}{\fill (\x,0) circle (1.5pt);}
  \node[font=\scriptsize,anchor=north] at (0,-0.08) {$\bar j$};
  \node[font=\scriptsize,anchor=north] at (1,-0.08) {$\bar i$};
  \node[font=\scriptsize,anchor=north] at (2.1,-0.08) {$i$};
  \node[font=\scriptsize,anchor=north] at (3.1,-0.08) {$j$};
  \draw[thick] (0,-0.05).. controls (0,-0.72) and (1,-0.72)..(1,-0.05);
  \draw[thick] (1,0.05).. controls (1,0.78) and (2.1,0.78)..(2.1,0.05);
  \draw[thick] (0,0.05).. controls (0,1.28) and (3.1,1.28)..(3.1,0.05);
  \node[font=\scriptsize,anchor=north] at (0.5,-0.70) {old top edge, copied below};
  \node[font=\scriptsize] at (2.05,1.45) {new top block};
  \node[font=\scriptsize,anchor=north] at (1.55,-1.12) {$i\;---\;\bar i\;---\;\bar j\;---\;j$};
\end{scope}
\end{tikzpicture}
\caption{The two operations on meanders.  In (a), endpoint extension attaches an alternating chain to the old top-unmatched endpoint, so a single path remains a single path.  In (b), an old top edge $i$--$j$ is replaced after folding by the three-edge path $i$--$\bar i$--$\bar j$--$j$.  The two outer edges belong to the new top block, while the middle edge is the reversed copy of the old top edge and lies in a bottom block.  The same pairing of the two copies produces the autocorrelation term in Proposition~\ref{prop:fold}.}
\label{fig:pathoperations}
\end{figure}

The figure also indicates the spectral role of the full fold.  The new $2N$-block pairs the potential histogram with its reversed copy, producing the autocorrelation term $(1+z)WW^\rv$ in the new Ooms multiplicity polynomial.

\begin{proposition}[Endpoint extension]\label{prop:extension}
For $r\ge0$, define
\[
 \s_r
 =\q((1,2^{[r]},a_2,\ldots,a_u)\mid(2^{[r]},B)).
\]
Then $\s_r$ is Frobenius. Its full histogram and Ooms multiplicity polynomial are
\begin{equation}\label{eq:extension}
 U_r=U+rV_\theta,
 \qquad
 H_{\s_r}=H+2r(1+z).
\end{equation}
\end{proposition}

\begin{proof}
Insert $2r$ new vertices before the old path. The new bottom pairs are $(2j-1,2j)$ and the new top pairs are $(2j,2j+1)$, $1\le j\le r$. They form the alternating chain shown in Figure~\ref{fig:pathoperations}(a), attached to the old endpoint, so the meander remains a single path. The corresponding potential is
\[
 (\theta,\theta-1,\theta,\ldots,\theta-1,\theta,\psi_2,\ldots,\psi_N).
\]
Each of the $2r$ new size-two directional blocks contributes $1+z$, while the old nonsingleton block contributions are unchanged. This proves \eqref{eq:extension}.
\end{proof}

For a composition $A=(a_1,\ldots,a_u)$, write $A^{\rm rev}=(a_u,\ldots,a_1)$.

\begin{proposition}[Full fold]\label{prop:fold}
Let $\q(A\mid B)\subseteq\slie_N$, $N\ge2$, be Frobenius, with full histogram $W$ and Ooms multiplicity polynomial $J$. Then
\[
 \mathcal D\q(A\mid B)=\q((2N)\mid(A^{\rm rev},B))
\]
is a Frobenius parabolic, and
\begin{equation}\label{eq:fold}
 H_{\mathcal D\q(A\mid B)}=J+(1+z)WW^\rv.
\end{equation}
\end{proposition}

\begin{proof}
Write the vertices of the folded meander as a reversed left copy and an ordinary right copy of the old vertex set.  If the old potential is $(w_1,\ldots,w_N)$, assign the folded potentials
\begin{equation}\label{eq:foldpotential}
 (w_N-1,\ldots,w_1-1\mid w_1,\ldots,w_N).
\end{equation}
Thus the left copy of the old vertex $i$ has potential $w_i-1$, while the right copy has potential $w_i$.

The new top composition has one block of size $2N$; its nested top arcs pair the two copies of each old vertex $i$.  Equation~\eqref{eq:foldpotential} therefore gives potential difference $1$ on every new top arc.  On the right, the bottom blocks are exactly the old $B$-blocks, so the old bottom arcs and their potential equations are unchanged.  On the left, reversing the order of the old $A$-blocks turns each old top pair $i<j$ into a bottom pair between the reversed copies of $j$ and $i$; its potential difference is
\[
 (w_j-1)-(w_i-1)=w_j-w_i=1,
\]
which is precisely the required bottom-arc equation.

The path structure is equally transparent from Figure~\ref{fig:pathoperations}(b).  If $i$ and $j$ were joined by an old top arc, then in the folded meander that single edge is replaced by the three-edge path
\[
 i\;---\;\overline{i}\;---\;\overline{j}\;---\;j,
\]
where bars denote vertices in the reversed left copy.  Old bottom edges remain unchanged in the right copy.  If an old vertex was unmatched on top, the new top block contributes one terminal edge joining it to its reversed copy.  Consequently the fold only subdivides old top edges and extends old top endpoints; it neither creates a cycle nor disconnects the graph.  Since the original meander is a single path, the folded meander is a single path as well.

It remains to compute the Ooms multiplicity polynomial.  The left bottom blocks reproduce, degree for degree, the contributions of the old top blocks, while the right bottom blocks reproduce the old bottom contributions.  Together they contribute exactly $J$.  The left half of the new even top block has potential histogram $z^{-1}W$.  Lemma~\ref{lem:evenblock} therefore gives the additional contribution
\[
 (1+z)(z^{-1}W)(zW^\rv)=(1+z)WW^\rv.
\]
Adding the two contributions proves \eqref{eq:fold}.
\end{proof}

Combining the two operations gives a useful first-difference criterion.

\begin{proposition}[Amplification criterion]\label{prop:amplification}
Fold $\s_r$ from Proposition~\ref{prop:extension}, and call the resulting parabolic $\p_r$. Then
\begin{equation}\label{eq:mastergeneral}
 H_{\p_r}
 =H+2r(1+z)+(1+z)(U+rV_\theta)(U^\rv+rV_\theta^\rv).
\end{equation}
Write
\[
 UU^\rv=\sum_tc_tz^t,
 \qquad
 UV_\theta^\rv+V_\theta U^\rv=\sum_tb_tz^t.
\]
For every $j\ge3$,
\begin{equation}\label{eq:ampdiff}
 \begin{split}
 [z^j]H_{\p_r}-[z^{j+1}]H_{\p_r}
 ={}&H_j-H_{j+1}+c_{j-1}-c_{j+1}\\
 &+r(b_{j-1}-b_{j+1}).
 \end{split}
\end{equation}
In particular, if $b_{j-1}<b_{j+1}$ for some $j\ge3$, then the spectra of $\p_r$ are nonunimodal for all sufficiently large $r$.
\end{proposition}

\begin{proof}
The only quadratic term in $r$ is
\[
 r^2(1+z)V_\theta V_\theta^\rv,
 \qquad
 V_\theta V_\theta^\rv=2+z+z^{-1},
\]
which is supported in degrees $-1,0,1,2$. The term $2r(1+z)$ is also central. For a reciprocal polynomial $C=\sum c_tz^t$,
\[
 [z^j](1+z)C-[z^{j+1}](1+z)C=c_{j-1}-c_{j+1}.
\]
Applying this to \eqref{eq:mastergeneral} proves \eqref{eq:ampdiff}.
\end{proof}

Thus a seed with a unimodal Ooms spectrum can generate an infinite counterexample family: the extension leaves all noncentral Ooms multiplicities fixed but changes the autocorrelation term introduced by the fold.

\section{Rectangular seeds}\label{sec:rectangular}

We now exhibit a two-parameter seed family for which every ingredient in the amplification formula can be written explicitly. Put
\[
 I_m=1+z+\cdots+z^{m-1},
 \qquad
 C_m=I_mI_m^\rv,
 \qquad
 B_m=C_m+zI_m.
\]

\begin{theorem}[Rectangular seed]\label{thm:seed}
For every $a,b\ge2$, the seaweed
\begin{equation}\label{eq:seedfamily}
 \s_{a,b}=\q((1,a-1,a(b-1))\mid(ab))
\end{equation}
is Frobenius. It admits a potential supported on $[0,a+b-2]$ whose full histogram and Ooms multiplicity polynomial are
\begin{equation}\label{eq:seedformulas}
 U_{a,b}=I_aI_b,
 \qquad
 H_{a,b}=B_{a-1}+C_aB_{b-1}.
\end{equation}
The Ooms spectrum is strictly unimodal, and
\[
 \dim\s_{a,b}=a(a-1)+a^2b(b-1).
\]
\end{theorem}

\begin{proof}
Use zero-based vertex labels $i=ak+r$ with $0\le k<b$ and $0\le r<a$, and set $n=ab$. Let $\tau$ and $\beta$ be the involutions defined by the top and bottom block pairings. Then
\[
 \tau(0)=0,
 \quad
 \tau(i)=a-i\ (1\le i<a),
 \quad
 \tau(i)=n+a-1-i\ (a\le i<n),
\]
and $\beta(i)=n-1-i$. Hence
\[
 (\beta\tau)(0)=n-1,
 \quad
 (\beta\tau)(i)=n-a-1+i\ (1\le i<a),
 \quad
 (\beta\tau)(i)=i-a\ (a\le i<n).
\]
In grid coordinates this decrements $k$ when $k>0$; at $k=0$ it returns to $k=b-1$ and decrements $r$ cyclically. Thus $\beta\tau$ is a single $ab$-cycle and the meander is connected. Across the two compositions there are exactly two odd parts, so the graph has $ab-1$ edges. Since every vertex has degree at most two, the connected graph is a single path.

For an explicit potential, define
\[
 \zeta_m(u)=
 \begin{cases}
 m-1-2u,&2u<m,\\
 2u-m,&2u\ge m,
 \end{cases}
\]
and put
\begin{equation}\label{eq:gridpotential}
 \psi_{ak+r}
 =\zeta_b(k)+
 \begin{cases}
 \zeta_a(r),&2k<b,\\
 \zeta_a(a-1-r),&2k\ge b.
 \end{cases}
\end{equation}
The elementary reflection identities
\[
 \zeta_m(m-1-u)-\zeta_m(u)=-\operatorname{sgn}(m-1-2u)
\]
and
\[
 \zeta_m(m-u)-\zeta_m(u)=\operatorname{sgn}(m-2u),\qquad1\le u<m,
\]
verify every top and bottom potential equation in \eqref{eq:potential}. For fixed $k$, the second summand in \eqref{eq:gridpotential} takes each value $0,\ldots,a-1$ once, while the first summands take $0,\ldots,b-1$ once. Hence $U_{a,b}=I_aI_b$.

The three top-block histograms are
\[
 X=z^{a+b-2},
 \qquad
 Y=z^{b-1}I_{a-1},
 \qquad
 Z=I_aI_{b-1}.
\]
Since the bottom composition has one part, direct triangular counting gives
\[
 H_{a,b}=XX^\rv+YY^\rv+ZZ^\rv+XY^\rv+XZ^\rv+YZ^\rv-1.
\]
Now
\[
 YY^\rv=C_{a-1},\quad
 ZZ^\rv=C_aC_{b-1},\quad
 XY^\rv=zI_{a-1},\quad
 (X+Y)Z^\rv=zI_{b-1}C_a,
\]
which gives \eqref{eq:seedformulas} and the dimension formula.

It remains to record strictness. The cross-block part is
\[
 F=zI_{a-1}+zI_{b-1}C_a.
\]
The remaining terms are reciprocal, and the symmetry $H=zH^\rv$ gives
\begin{equation}\label{eq:seeddiff}
 H_j-H_{j+1}
 =\one_{j\le a-1}
 +\sum_{v=1}^{b-1}\bigl((a-|j-v|)_+-(a-j-v)_+\bigr).
\end{equation}
Every summand is nonnegative. For $1\le j\le a+b-2$, some $v\in[1,b-1]$ satisfies $|j-v|\le a-1$, and the corresponding summand is positive. Thus every outward first difference is positive.
\end{proof}

Apply the extension and fold of Section~\ref{sec:amplifier} to \eqref{eq:seedfamily}. Since the first endpoint has potential $D=a+b-2$, put
\[
 V=z^{D-1}+z^D.
\]
We obtain exactly the parabolics \eqref{eq:introfamily}.

\begin{corollary}\label{cor:masterrect}
Every $\p_{a,b,r}$ in \eqref{eq:introfamily} is Frobenius, and
\begin{equation}\label{eq:masterrect}
 \boxed{
 H_{a,b,r}
 =B_{a-1}+C_aB_{b-1}+2r(1+z)
 +(1+z)(I_aI_b+rV)(I_a^\rv I_b^\rv+rV^\rv).}
\end{equation}
Its support is $[-D,D+1]$, where $D=a+b-2$, and
\begin{equation}\label{eq:dimrect}
 \dim\p_{a,b,r}
 =8r^2+(8ab+4)r+a(a-1)+a^2b(3b-1).
\end{equation}
\end{corollary}

\section{Arbitrarily long outward rises}\label{sec:long}

We now specialize to square seeds, where the correlation terms in \eqref{eq:masterrect} reduce to cubic polynomials.

Let $a=b=t\ge6$, so $U=I_t^2$ and $D=2t-2$. Write
\[
 UU^\rv=\sum_dc_dz^d.
\]
For $0\le d\le t$, coefficient extraction from
\[
 I_t^4=z^0\frac{(1-z^t)^4}{(1-z)^4}
\]
gives
\begin{equation}\label{eq:spline}
 c_d
 =\binom{2t+1-d}{3}-4\binom{t+1-d}{3}
 =\frac{2t^3+t}{3}-td^2+\frac{d^3-d}{2}.
\end{equation}
In particular,
\begin{equation}\label{eq:corrdiff}
 c_{j-1}-c_{j+1}=4jt-3j^2.
\end{equation}

Formula \eqref{eq:seeddiff}, summed explicitly for $a=b=t$, gives
\begin{equation}\label{eq:seedclosed}
 H_j-H_{j+1}
 =2jt-\frac{3j^2+j}{2}+1,
 \qquad1\le j\le t-1.
\end{equation}
Finally, for the mixed correlation
\[
 M=UV^\rv+VU^\rv,
\]
reciprocity of $U$ gives
\begin{equation}\label{eq:mixed}
 [z^d]M=U_d+U_{d+1}=2d+3,
 \qquad2\le d\le t-2.
\end{equation}
The quadratic term in $r$ is central. Combining \eqref{eq:corrdiff}, \eqref{eq:seedclosed}, and \eqref{eq:mixed} yields the exact formula announced in the introduction.

\begin{theorem}\label{thm:longgap}
For $t\ge6$ and $3\le j\le t-3$,
\[
 h_j(t,t,r)-h_{j+1}(t,t,r)
 =6jt-\frac{9j^2+j}{2}+1-4r.
\]
\end{theorem}

\begin{proof}
The fixed seed contributes \eqref{eq:seedclosed}, the fold autocorrelation contributes \eqref{eq:corrdiff}, and the mixed term contributes
\[
 r\bigl(M_{j-1}-M_{j+1}\bigr)=-4r.
\]
The terms quadratic in $r$, as well as the direct extension term, are supported at the center and do not enter these degrees.
\end{proof}

\begin{proof}[Proof of Theorem~A]
For $3\le j\le t-3$,
\[
 6jt-\frac{9j^2+j}{2}+1
 \le 6jt-\frac92j^2+1
 \le 2t^2+1.
\]
At $r=t^2$, Theorem~\ref{thm:longgap} therefore gives
\begin{equation}\label{eq:uniformnegative}
 h_j-h_{j+1}\le-(2t^2-1)
 \qquad(3\le j\le t-3).
\end{equation}
Take $t=L+5$.  Then the indicated range contains exactly $L$ indices, $j=3,\ldots,L+2$, proving \eqref{eq:strongdefectintro} and the consecutive rise.

Summing the exact negative of Theorem~\ref{thm:longgap} over this range gives
\[
 h_{t-2}-h_3
 =\sum_{j=3}^{t-3}(h_{j+1}-h_j)
 =\frac{(t-5)(5t^2-7t+16)}2.
\]
Substituting $t=L+5$ gives \eqref{eq:massintro}.  Finally, for any $e\in\g_1$ the grading implies
\[
 \ad e(\g_j)\subseteq\g_{j+1}.
\]
Hence
\[
 \dim\operatorname{coker}(\ad e:\g_j\to\g_{j+1})
 \ge h_{j+1}-h_j\ge2t^2-1
\]
on each of the $L$ steps.  The corresponding algebra lies in $\slie_{6t^2}$.
\end{proof}

\begin{corollary}[Unbounded $SL_2$ defect]\label{cor:sl2defect}
For a semisimple integral Frobenius principal grading define
\[
 \operatorname{Def}_{SL_2}(\g)
 =\sum_{m\ge1}\max\{0,h_{m+1}-h_m\}.
\]
This is the total negative coefficient mass of the virtual character \eqref{eq:weylexpansion}, and it vanishes exactly when the Ooms multiplicities are unimodal.  Along the family in Theorem~A,
\[
 \operatorname{Def}_{SL_2}(\g)
 \ge \frac{L(5L^2+43L+106)}2,
\]
so the defect is unbounded and grows at least cubically in the length of the prescribed outward rise.
\end{corollary}

The same formulas show that the outward rise can have unbounded relative height.

\begin{corollary}\label{cor:ratio}
For every $R>0$, there is a Frobenius parabolic in the family \eqref{eq:introfamily} and an integer $j>3$ such that
\[
 h_j>Rh_3.
\]
\end{corollary}

\begin{proof}
For $3\le j\le t-2$, the coefficient of $r$ in $h_j(t,t,r)$ is $4(j+1)$ and the quadratic coefficient vanishes. Therefore
\[
 \lim_{r\to\infty}\frac{h_{t-2}(t,t,r)}{h_3(t,t,r)}=\frac{t-1}{4}.
\]
Choose $t$ so that $(t-1)/4>R$, then take $r$ sufficiently large.
\end{proof}

The corollary concerns the ratio of two positive-degree multiplicities. It is not a statement about their proportions in the total dimension, nor does Theorem~A assert arbitrarily many distinct local maxima.  The centered principal character can therefore have an arbitrarily long block of negative irreducible $SL_2$ multiplicities, with unbounded total negative mass.

\section{A simple one-parameter family}\label{sec:simple}

Set $(a,b)=(6,5)$ in Corollary~\ref{cor:masterrect}. Then
\[
 \s_{6,5}=\q((1,5,24)\mid(30)),
 \qquad
 U=I_6I_5,
\]
with coefficient list
\[
 (1,2,3,4,5,5,4,3,2,1).
\]
The resulting parabolics are precisely \eqref{eq:simplefamily}. Expanding \eqref{eq:masterrect} gives the complete positive half of the spectrum.

\begin{table}[htbp]
\centering
\caption{Positive-degree multiplicities for the family $\p_r$ in \eqref{eq:simplefamily}.}
\label{tab:simple}
\begin{tabular}{rll}
\toprule
$j$ & $h_j(r)$ & $h_j(r)-h_{j+1}(r)$\\
\midrule
1&$310+14r+3r^2$&$27+r+2r^2$\\
2&$283+13r+r^2$&$46-3r+r^2$\\
3&$237+16r$&$56-3r$\\
4&$181+19r$&$57$\\
5&$124+19r$&$49+3r$\\
6&$75+16r$&$35+4r$\\
7&$40+12r$&$22+4r$\\
8&$18+8r$&$12+4r$\\
9&$6+4r$&$5+3r$\\
10&$1+r$&$1+r$\\
\bottomrule
\end{tabular}
\end{table}

The second difference column is positive for every $r\ge0$ except possibly the entry $56-3r$: indeed
\[
 r^2-3r+46=(r-1)(r-2)+44>0.
\]
Thus the transition occurs exactly between $r=18$ and $r=19$. The dimension formula \eqref{eq:dimrect} becomes
\[
 \dim\p_r=8r^2+244r+2550.
\]
At $r=19$ this is $10{,}074$.

Figure~\ref{fig:histograms} displays the full multiplicity sequence before the extension begins and at the first nonunimodal parameter. The symmetry relation is $h_{1-j}=h_j$.

\begin{figure}[htbp]
\centering
\begin{tikzpicture}
\begin{axis}[
 width=0.96\textwidth,height=4.25cm,
 ybar,bar width=5pt,
 xmin=-9.7,xmax=10.7,
 xtick={-9,-8,...,10},
 tick label style={font=\scriptsize},
 xlabel={eigenvalue $\lambda$},ylabel={$h_\lambda$},
 title={$r=0$: strictly unimodal, $\dim\p_0=2550$},
 title style={font=\small},
 axis x line*=bottom,axis y line*=left,
 enlarge y limits={upper,value=0.10},
]
\addplot+[draw=black,fill=black!20] coordinates {
(-9,1)(-8,6)(-7,18)(-6,40)(-5,75)(-4,124)(-3,181)(-2,237)(-1,283)(0,310)
(1,310)(2,283)(3,237)(4,181)(5,124)(6,75)(7,40)(8,18)(9,6)(10,1)};
\end{axis}
\end{tikzpicture}

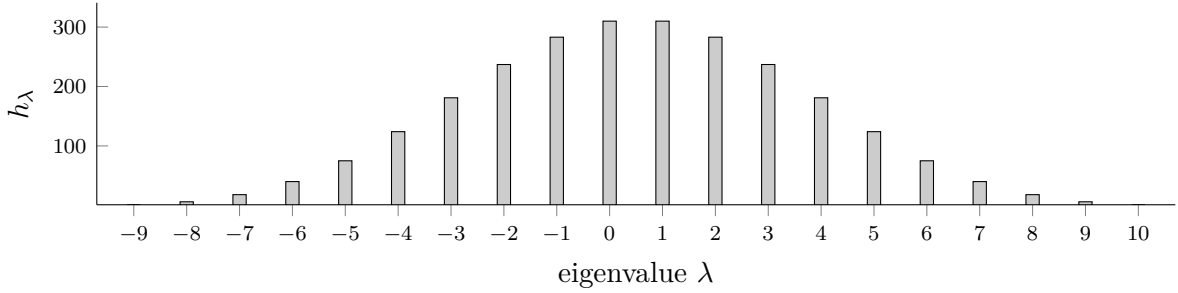
\vspace{0.6em}

\begin{tikzpicture}
\begin{axis}[
 width=0.96\textwidth,height=4.25cm,
 ybar,bar width=5pt,
 xmin=-9.7,xmax=10.7,
 xtick={-9,-8,...,10},
 tick label style={font=\scriptsize},
 xlabel={eigenvalue $\lambda$},ylabel={$h_\lambda$},
 title={$r=19$: first nonunimodal member, $\dim\p_{19}=10074$},
 title style={font=\small},
 axis x line*=bottom,axis y line*=left,
 enlarge y limits={upper,value=0.14},
]
\addplot+[draw=black,fill=black!20] coordinates {
(-9,20)(-8,82)(-7,170)(-6,268)(-5,379)(-4,485)(-3,542)(-2,541)(-1,891)(0,1659)
(1,1659)(2,891)(3,541)(4,542)(5,485)(6,379)(7,268)(8,170)(9,82)(10,20)};
\node[font=\scriptsize,anchor=south west] at (axis cs:2.25,620) {$h_3=541<h_4=542$};
\draw[->,thin] (axis cs:3.0,610) -- (axis cs:3,548);
\draw[->,thin] (axis cs:4.35,610) -- (axis cs:4,549);
\end{axis}
\end{tikzpicture}
\caption{Full Ooms multiplicity histograms for two members of the same Frobenius parabolic family. The horizontal axis is the eigenvalue. The lower plot contains the first outward reversal, at eigenvalues $3$ and $4$.}
\label{fig:histograms}
\end{figure}
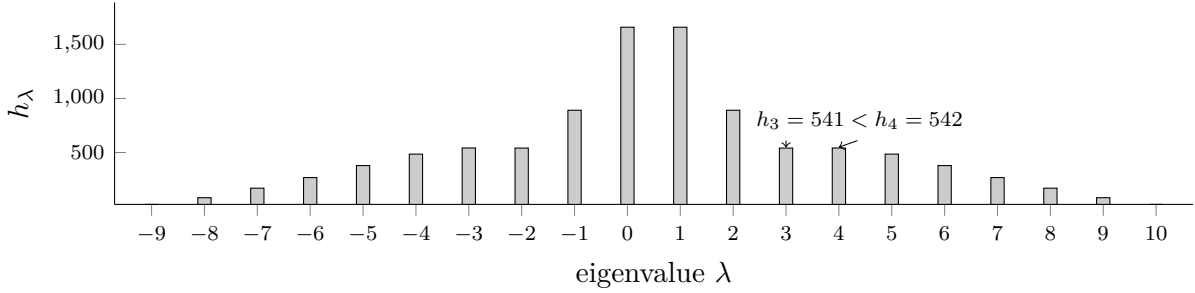

By Theorem~\ref{thm:sl2dictionary}, the first differences at $r=19$ are the irreducible coefficients of the centered principal character:
\[
 (768,350,-1,57,106,111,98,88,62,20).
\]
Thus
\begin{equation}\label{eq:136virtual}
 \Theta_{\p_{19}}
 =768\chi_1+350\chi_3-\chi_5+57\chi_7+106\chi_9+111\chi_{11}
 +98\chi_{13}+88\chi_{15}+62\chi_{17}+20\chi_{19}.
\end{equation}
The unique negative coefficient is $-1$ at $\chi_5$.

The full spectrum at $r=19$ is therefore
\[
\begin{array}{c|rrrrrrrrrr}
\lambda&-9&-8&-7&-6&-5&-4&-3&-2&-1&0\\
\hline
h_\lambda&20&82&170&268&379&485&542&541&891&1659
\end{array}
\]
\[
\begin{array}{c|rrrrrrrrrr}
\lambda&1&2&3&4&5&6&7&8&9&10\\
\hline
h_\lambda&1659&891&541&542&485&379&268&170&82&20.
\end{array}
\]
This proves Theorem~B.

\begin{remark}\label{rem:minimality}
An exact finite check over the three-parameter family \eqref{eq:introfamily} shows that $(a,b,r)=(6,5,19)$ has the smallest matrix size and the smallest Lie-algebra dimension among nonunimodal members of that construction. The check is not an enumeration of all Frobenius parabolics, and no global minimality assertion is made.
\end{remark}

\section{Boundary questions}

The $SL_2$ formulation raises a structural question.  Maximal parabolics have effective centered principal characters, while the examples here have virtual characters with negative irreducible coefficients.  One problem is to characterize this boundary intrinsically. In particular, it would be useful to determine the least number of parabolic blocks for which nonunimodality can occur, and to classify the endpoint histograms for which the linear term in Proposition~\ref{prop:amplification} has negative first difference.

The criterion in Proposition~\ref{prop:amplification} can also be stated in terms of the endpoint histogram. Given a Frobenius seaweed with a distinguished endpoint, the sign pattern of
\[
 UV_\theta^\rv+V_\theta U^\rv
\]
controls whether repeated extension followed by folding eventually destroys unimodality away from the center. The rectangular seeds give an explicit family with this behavior.  It remains to determine how frequently the same sign pattern occurs for other Frobenius meanders.

\appendix
\section{An independent exact certificate for the \texorpdfstring{$\slie_{136}$}{sl(136)} example}\label{app:certificate}

The proof of Theorem~B is already complete from the single-path criterion and the exact polynomial formula. We record a direct Kirillov certificate for the first concrete counterexample as an independent check.

For $r=19$, let $F$ be the directed-arc functional on
\[
 \p_{19}=\q((136)\mid(24,5,2^{[19]},1,2^{[19]},30)).
\]
Using the allowed off-diagonal matrix units together with
\[
 D_i=E_{ii}-E_{136,136},\qquad1\le i\le135,
\]
as a basis, the Kirillov form decomposes by principal degree. Since $F$ is supported in degree $1$, only degrees $k$ and $1-k$ pair. Exact elimination gives the following block determinants.

\begin{center}
\begin{tabular}{r|rrrrrrrrrr}
\toprule
$k$&$-9$&$-8$&$-7$&$-6$&$-5$&$-4$&$-3$&$-2$&$-1$&$0$\\
\midrule
block size&20&82&170&268&379&485&542&541&891&1659\\
$|\det M_k|$&1&1&1&1&1&1&1&1&1&136\\
\bottomrule
\end{tabular}
\end{center}

A rational elimination and a separate integral unit-pivot elimination agree. In the central block, after $1657$ unit pivots, the latter leaves
\[
 \begin{pmatrix}79&-78\\57&-58\end{pmatrix},
 \qquad\det=-136.
\]
Therefore
\[
 |\det B_F|=136^2=18{,}496\ne0.
\]
Every supporting arc has principal degree one, so the corresponding trace-zero diagonal satisfies $F([\widehat F,x])=F(x)$ on the matrix-unit basis. This verifies directly, without appealing to the meander-index theorem, that the displayed $541<542$ reversal is an Ooms-spectrum reversal for a Frobenius parabolic.

The exact-arithmetic verification used to audit the finite example and the restricted minimum in Remark~\ref{rem:minimality} is available from the author. These calculations are checks on finite data; none of the infinite statements above is inferred from a numerical search.

\end{document}